\documentclass{article}
  \usepackage[left=1.25in,right=1.25in,top=1.2in,bottom=1.2in]{geometry}
  \usepackage{xcolor}
\usepackage{amsmath,amssymb,amsthm}
\usepackage{microtype}
\usepackage{hyperref}
\hypersetup{
  pdftitle={Algebraic Vector Fields on S2 are 2-generated},
  pdfauthor={Yiyang Jiang and Xudong Chen}}
  \usepackage[style=numeric-comp,hyperref=true,doi=false,url=false,isbn=false,giveninits=true,sorting=none,maxnames=99]{biblatex}
  \DeclareFieldFormat[article,online,inbook,incollection,inproceedings,thesis]{title}{\mkbibemph{#1}}

\DeclareFieldFormat{eprint:arxiv}{%
  \ifhyperref
    {\href{https://arxiv.org/abs/#1}{arXiv\addcolon\addnbspace\mbox{#1}}}
    {arXiv\addcolon\addnbspace\mbox{#1}}}
    
\newcommand{\R}{\mathbb R}
\newcommand{\X}{\mathfrak X_{\mathrm{alg}}}
\newcommand{\Lie}{\operatorname{Lie}}
\newcommand{\dv}{\operatorname{div}}
\newcommand{\curl}{\operatorname{curl}}
\newcommand{\sph}{S^2}

\newcommand{\OO}{\mathcal O}
\newcommand{\LL}{\mathcal L}
\newcommand{\WW}{\mathcal W}
\newcommand{\ad}{\operatorname{ad}}

\newtheorem{lemma}{Lemma}
\newtheorem{remark}{Remark}
\newtheorem{proposition}[lemma]{Proposition}
\newtheorem{controlcor}{Corollary}

\newtheorem{mainthm}{Theorem}

\title{Algebraic Vector Fields on $S^2$ are $2$-generated}
  \author{Yiyang Jiang$^*$ \quad and \quad Xudong Chen\footnote{Y. Jiang and X. Chen are with the Department of Electrical \& Systems Engineering, Washington University in St.~Louis, St.~Louis, MO 63130, USA. Emails: {\small\texttt{j.yiyang@wustl.edu, cxudong@wustl.edu}}. Corresponding author: Y. Jiang.}}
  \date{}
  
\begin{document}
\maketitle

\begin{abstract}
We exhibit two vector fields on $S^2$, which generate the entire Lie algebra of algebraic vector fields. 
Compositions of the flows associated with these two vector fields enable us to approximate every diffeomorphism of $S^2$ isotopic to the identity in the $\mathrm{C}^r$-sense for any $r\geq 0$.
\end{abstract}

\section{Introduction and main results}\label{sec:intro}

Let $\sph:=\{x\in\R^3\mid  \|x\|_2=1\}$ be the unit sphere embedded in $\R^3$.  Consider the following control-linear system:
\begin{equation}\label{eq:system}
 \dot x(t)= \sum_{i = 1}^m u_i(t)W_i(x(t)),\qquad x\in\sph,
\end{equation}
where each $W_i$ is a smooth vector field on $S^2$ and each $u_i(t)$ is a scalar control input. 
We address (approximate) controllability of system~\eqref{eq:system} on the identity component of the group of diffeomorphisms of $S^2$.
The problem has been addressed in~\cite{AS22}, where the authors have shown that approximate controllability can be achieved if one is allowed to use seven vector fields (i.e., $m = 7$).  
In this paper, we sharpen their result by showing that one only needs two vector fields to attain the goal. 

We now formulate the main result precisely. 
First, a vector field on $S^2$ is said to be {\it algebraic} if it is obtained by restricting a polynomial vector field on $\R^3$ to $S^2$.  
We use $\X(\sph)$ to denote the Lie algebra (over $\R$) of these fields.
Given algebraic vector fields $U_1,\ldots,U_m$ on $S^2$, we use $\Lie_{\R}\{U_1,\ldots,U_m\}$ to denote the Lie subalgebra of $\X(\sph)$ generated by $U_1,\ldots, U_m$. 
Next, to a given polynomial $f\in\R[x_1,x_2,x_3]$, we associate two vector fields on $S^2$ as follows: 
\begin{equation}\label{eq:fields}
 \begin{aligned}
 X_f&:=x\times\nabla f,\\
 G_f&:=\nabla f-(x\cdot\nabla f)x,
 \end{aligned}
\end{equation}
where ``$\times$'' and ``$\cdot$'' are the cross-product and the dot-product, respectively, and $\nabla f$ is the gradient of~$f$.
Note, in particular, that $G_f$ is the gradient of~$f$ (with respect to the metric induced by the standard metric in $\R^3$). 
With the preliminaries above, we now state the first main result of the paper.

\begin{mainthm}\label{thm:main}
Let
\begin{equation}\label{eq:generators}
 \begin{aligned}
 W_1&:=X_{x_3},\\
 W_2&:=X_{x_1}+X_{x_1^2-x_2^2}+G_{x_3}.
 \end{aligned}
\end{equation}
Then $$\Lie_{\R}\{W_1,W_2\}=\X(\sph).$$
\end{mainthm}

The problem of whether the Lie algebra of algebraic vector fields on a real, smooth manifold is finitely generated by complete vector fields has recently been addressed in~\cite{Real}. In that work, we have shown that, for any real, connected, semi-simple matrix group, its Lie algebra of algebraic vector fields can be generated by three complete vector fields. The same problem, but for complex manifolds $M$ (with the requirement that the generating vector fields are $\mathbb{C}$-complete), has been addressed in the literature. 
For $\mathbb{C}^d$ with $d\ge2$, it has been shown~\cite{BP26} that $\X(\mathbb{C}^d)$ is $2$-generated. 
For a Danielewski surface $M$ of the form $xy = p(z)$, where p(z) is a polynomial with simple zeros, it has been shown~\cite{And26} that $\X(M)$ is $6$-generated. 
Finally, for a broader class of classical complex Lie groups $G$, we have shown in~\cite{Complex} that $\X(G)$ is $3$-generated.  

Given an arbitrary smooth vector field $U$ on $S^2$, let $\Phi_U^t$ be the associated flow. 
Given the two vector fields $W_1$ and $W_2$ in Theorem~\ref{thm:main}, let $\WW$ be the group generated by the flows $\Phi^t_{W_i}$ for $i = 1,2$ and for $t\in \R$, i.e., $\WW$ comprises the following  finite compositions:
\begin{equation}\label{eq:finitecomposition}
 \Phi_{W_{i_N}}^{t_N}\circ\cdots\circ\Phi_{W_{i_1}}^{t_1}, \qquad N\geq1,\quad i_j\in\{1,2\},\quad t_j\in\R.
\end{equation}
Note that every diffeomorphism in $\WW$ can be realized by system~\eqref{eq:system} with $m = 2$; indeed, given the composition in~\eqref{eq:finitecomposition}, we simply let $u_{i_j}(t)=\operatorname{sgn}(t_j)$ and the other control input be zero for the time period $[\sum_{s=1}^{j-1}|t_s|, \sum_{s=1}^{j}|t_s|)$, for any $j=1,\dots,N$.
The following result is a corollary of Theorem~\ref{thm:main}: 

\begin{controlcor}\label{cor:control}
For any diffeomorphism $\Psi$ of $\sph$ that is smoothly isotopic to the identity, any integer $r\geq0$, and any $\varepsilon>0$, there is a $\phi\in\WW$ such that $\|\phi-\Psi\|_{\mathrm{C}^r(\sph,\R^3)}<\varepsilon$.
\end{controlcor}

The remainder of the paper is organized as follows:
Section~\ref{sec:preliminaries} introduces preliminaries that are key to the proofs of the main results,  
Section~\ref{sec:PfThm} proves Theorem~\ref{thm:main}, and Section~\ref{sec:PfCor} proves Corollary~\ref{cor:control}.
Section~\ref{sec:conclusions} concludes the paper.

\section{Preliminaries}\label{sec:preliminaries}

In this section, we gather a few preliminary results that will be used in the proof of Theorem~\ref{thm:main}. 

\subsection{Polynomial functions and spherical harmonics}
We first note that a polynomial vanishes on $\sph$ if and only if it is a multiple of $x\cdot x-1$.
To see this, divide such a polynomial by $x\cdot x-1$, which is monic in $x_3$; the remainder has the form $a(x_1,x_2)x_3+b(x_1,x_2)$.
Evaluating at $x_3=\pm\sqrt{1-x_1^2-x_2^2}$ shows that $a$ and $b$ vanish on the open disk $x_1^2+x_2^2<1$, so both are zero.
It follows that the algebra of polynomial functions on $\sph$ is $\OO:=\R[x_1,x_2,x_3]/(x\cdot x-1)$.
The fields in~\eqref{eq:fields} depend only on the restriction of $f$ to $\sph$.
Indeed, replacing $f$ by $f+(x\cdot x-1)h$ changes $\nabla f$ on $\sph$ by $2hx$, which contributes nothing to $X_f$ because $x\times x=0$, and nothing to $G_f$ because $x\cdot x=1$.
We may therefore write $X_f$ and $G_f$ for $f\in\OO$.

By the definition of algebraic vector fields,
\begin{equation}\label{eq:tangent}
 \X(\sph)=\{U\in\OO^3\mid x\cdot U=0\}.
\end{equation}
For a smooth vector field $U$ on $\sph$ and a smooth function $f$, we write $U(f)$ for the derivative of $f$ along $U$, and we adopt the convention $[U,V](f)=U(V(f))-V(U(f))$.

For each $k\geq0$, let $H_k$ be the space of restrictions to $\sph$ of homogeneous polynomials of degree $k$ that are annihilated by the Euclidean Laplacian $\Delta:= \sum_{i = 1}^3 \frac{\partial^2}{\partial x_i^2}$.
We recall the spherical-harmonic decomposition of $\OO$ and the eigenvalues of the spherical Laplacian $\Delta_{\sph}$ (see~\cite[Theorems~1.1.2 and~1.1.3, Corollary~1.1.4, and Theorem~1.4.5]{DX13}):
\begin{equation}\label{eq:harmonics}
 \begin{aligned}
 \OO&=\bigoplus_{k\geq0}H_k,\qquad \dim H_k=2k+1,\\
 \Delta_{\sph}f&=-k(k+1)f\qquad \mbox{for any } f\in H_k.
 \end{aligned}
\end{equation}
Note that each polynomial function has only finitely many nonzero components in this decomposition.
The space $H_0$ consists of the constants, and $H_1=\operatorname{span}_{\R}\{x_1,x_2,x_3\}$.
For $R\in\mathrm{SO}(3)$ and $f\in H_k$, the composition $f\circ R$ is again an element of $H_k$; this holds because $\Delta(p\circ R)=(\Delta p)\circ R$ for every polynomial $p$.
Moreover, $H_k$ is irreducible, meaning that if a subspace $E\subseteq H_k$ contains $f\circ R$ for all $f\in E$ and $R\in\mathrm{SO}(3)$, then $E=\{0\}$ or $E=H_k$; see~\cite[Theorem~1.7.2]{DX13}.
This irreducibility is used only in the proof of Proposition~\ref{prop:criterion}.

\subsection{Hamiltonian and gradient fields}
We write $X_{H_k}:=\{X_f\mid f\in H_k\}$ 
and $G_{H_k}:=\{G_f\mid f\in H_k\}$. 
The space $X_{H_1}=\{x\times a\mid a\in\R^3\}$ consists of the rotation fields.
For $R\in\mathrm{SO}(3)$, the chain rule and the identity $R(a\times b)=Ra\times Rb$ imply that $$X_{f\circ R}(x)=R^\top X_f(Rx) \quad \mbox{and} \quad G_{f\circ R}(x)=R^\top G_f(Rx).$$ In words, rotating $f$ rotates $X_f$ and $G_f$.
For $k\ge1$, the maps $f\mapsto X_f$ and $f\mapsto G_f$ are injective on $H_k$; this holds because $X_f=0$ or $G_f=0$ forces $f$ to be constant.

Let $\omega_x(a,b):=x\cdot(a\times b)$ be the standard area form on $\sph$, and let $\dv$ denote divergence with respect to $\omega$.
The field $X_f$ is Hamiltonian with respect to $\omega$, with the convention $\iota_{X_f}\omega=-df$.
For a smooth vector field $U$ on $\sph$, we define its scalar surface curl by $\curl U:=-\dv(x\times U)$.
The two families in~\eqref{eq:fields} then satisfy
\begin{equation}\label{eq:divergence}
 \begin{aligned}
 \dv X_f&=0,&\dv G_f&=\Delta_{\sph}f,\\
 \curl X_f&=\Delta_{\sph}f,&\curl G_f&=0.
 \end{aligned}
\end{equation}
The first identity holds because, by Cartan's formula, the Lie derivative of $\omega$ along $X_f$ is $d\iota_{X_f}\omega=-d(df)=0$; the second is the definition of the spherical Laplacian.
The curl identities follow from $x\times X_f=-G_f$ and $x\times G_f=X_f$.

The next proposition, a polynomial form of the Helmholtz decomposition, shows that every algebraic vector field splits into the two families.

\begin{proposition}\label{prop:decomposition}
Every algebraic vector field $U$ on $\sph$ has a unique finite decomposition
\begin{equation}\label{eq:XplusG}
 U=\sum_{k\geq1}\bigl(X_{p_k}+G_{q_k}\bigr),\qquad p_k,q_k\in H_k.
\end{equation}
Consequently,
\begin{equation}\label{eq:XplusGspace}
 \begin{aligned}
 \X(\sph)&=\bigoplus_{k\geq1}\bigl(X_{H_k}\oplus G_{H_k}\bigr),\\
 \ker\dv&=\bigoplus_{k\geq1}X_{H_k}.
 \end{aligned}
\end{equation}
\end{proposition}

\begin{proof}
For uniqueness, apply both $\curl$ and $\dv$ to~\eqref{eq:XplusG}.
By~\eqref{eq:harmonics} and~\eqref{eq:divergence}, the degree-$k$ harmonic components of $\curl U$ and $\dv U$ are ${-k(k+1)p_k}$ and ${-k(k+1)q_k}$, which determine all $p_k$ and $q_k$.
Hence the sum of the spaces $X_{H_k}$ and $G_{H_k}$, $k\geq1$, is direct, and each of these spaces has dimension $2k+1$.

For existence, we count dimensions.
Let $\OO_{\leq k}$ be the space of restrictions to $\sph$ of polynomials of degree at most $k$.
Since a polynomial of degree at most $k$ that vanishes on $\sph$ has the form $(x\cdot x-1)h$ with $\deg h\leq k-2$, we have $\dim\OO_{\leq k}=\tbinom{k+3}{3}-\tbinom{k+1}{3}=(k+1)^2$.
For $k\geq1$, the map $U\mapsto x\cdot U$ sends $\OO_{\leq k}^3$ onto $\OO_{\leq k+1}$.
Indeed, the constant $1$ is the image of $x$, and a homogeneous polynomial $f$ of degree $1\leq j\leq k+1$ satisfies $x\cdot\nabla f=jf$, so it is the image of $\nabla f/j\in\OO_{\leq k}^3$.
By~\eqref{eq:tangent}, the kernel of this map is $\mathfrak X_{\leq k}:=\X(\sph)\cap\OO_{\leq k}^3$, which therefore has dimension $3(k+1)^2-(k+2)^2=2k^2+2k-1$.
On the other hand, for $f\in H_j$, the fields $X_f$ and $G_f$ computed from its homogeneous representative have coefficients of degree at most $j$ and $j+1$, respectively, so
$$
 \bigoplus_{j=1}^{k}X_{H_j}\oplus\bigoplus_{j=1}^{k-1}G_{H_j}\subseteq\mathfrak X_{\leq k}.
$$
The left side has dimension $(k+1)^2-1+k^2-1=2k^2+2k-1$, so equality holds.
Since every algebraic vector field lies in some $\mathfrak X_{\leq k}$, this proves~\eqref{eq:XplusG} and the first line of~\eqref{eq:XplusGspace}.
Finally, $\dv U=-\sum_k k(k+1)q_k$ vanishes exactly when all $q_k=0$, which proves the second line.
\end{proof}

We now compute brackets.
For $f,g\in\OO$, set $\{f,g\}:=x\cdot(\nabla f\times\nabla g)=X_f(g)$. 
This is the Poisson bracket on $\sph$ associated with $\omega$, since $\{f,g\}=\omega(X_f,X_g)$.
It is well defined on $\OO$, and $\{x_1,x_2\}=x_3$ together with its cyclic permutations.
Since the bracket satisfies the Leibniz rule in each argument, the Jacobi identity needs to be checked only on $x_1,x_2,x_3$, where it follows from the cyclic relations. Hence $[X_f,X_g](h)=\{f,\{g,h\}\}-\{g,\{f,h\}\}=\{\{f,g\},h\}$ for every $h\in\OO$, that is,
\begin{equation}\label{eq:XX}
 [X_f,X_g]=X_{\{f,g\}}.
\end{equation}
In particular, the Hamiltonian fields $X_f$, $f\in\OO$, form a Lie subalgebra of $\X(\sph)$, and by~\eqref{eq:XplusGspace} this subalgebra is $\ker\dv$.

For $\ell\in H_1$, the flow of $X_\ell$ consists of rotations $R_t$.
Differentiating $G_{f\circ R_t}(x)=R_t^\top G_f(R_tx)$ at $t=0$, we obtain $G_{\{\ell,f\}}$ on the left and $[X_\ell,G_f]$ on the right, so
\begin{equation}\label{eq:XG}
 [X_{\ell},G_f]=G_{\{\ell,f\}}.
\end{equation}
We also have
\begin{equation}\label{eq:GG}
 [G_f,G_g]=-X_{\{f,g\}},\quad \mbox{for } f,g\in H_1.
\end{equation}
To see this, write $f=a\cdot x$ and $g=b\cdot x$.
Then $G_f=a-fx$ and $G_g=b-gx$, and a direct computation gives $[G_f,G_g]=fb-ga=-x\times(a\times b)=-X_{\{f,g\}}$.

\section{Proof of Theorem~\ref{thm:main}}\label{sec:PfThm}

The proof has two steps.
In the first step, we show that if a set contains a basis of the space of rotation fields, together with two other low-degree vector fields, then the Lie algebra generated by the set is $\X(\sph)$. This is done in
Proposition~\ref{prop:criterion}. Then, in the second step,  we show that the Lie algebra generated by $\{W_1, W_2\}$ contains that set. 

\subsection{On finiteness of generators}
The main result of this subsection is the following proposition, which in particular implies that $\X(S^2)$ is finitely generated. 

\begin{proposition}\label{prop:criterion}
Let $\LL\subseteq\X(\sph)$ be a Lie subalgebra containing all rotation fields, a nonzero field in $X_{H_2}$, and a nonzero field in $G_{H_1}$.
Then, $\LL=\X(\sph)$.
\end{proposition}

\begin{proof}
Fix $k\ge1$ and let $E:=\{f\in H_k\mid X_f\in\LL\}$.
For $\ell\in H_1$, the flow $R_t$ of $X_\ell$ consists of rotations, and for $f\in H_k$ the function $g(t):=f\circ R_t$ solves the linear equation $\dot g=\{\ell,g\}$ on $H_k$.
If $f\in E$, then~\eqref{eq:XX} gives $X_{\{\ell,f\}}=[X_\ell,X_f]\in\LL$, so $E$ is invariant under the map $f\mapsto\{\ell,f\}$.
It follows that $f\circ R_t\in E$ for all $f\in E$ and $t\in\R$.
Since every rotation is the time-one flow of some rotation field, $E$ contains $f\circ R$ for all $f\in E$ and $R\in\mathrm{SO}(3)$, and irreducibility gives $E=\{0\}$ or $E=H_k$.
The same argument, with~\eqref{eq:XG} in place of~\eqref{eq:XX}, applies to $\{f\in H_k\mid G_f\in\LL\}$.
We conclude that, for each $k\ge1$, if $\LL$ contains a nonzero field in $X_{H_k}$ or $G_{H_k}$, then it contains that whole space.
In particular, the hypotheses give $X_{H_2}\subseteq\LL$ and $G_{H_1}\subseteq\LL$.

We use complex-valued polynomials only to shorten the following calculations.
For real polynomials $p,q$, we set $X_{p+iq}:=X_p+iX_q$ and $G_{p+iq}:=G_p+iG_q$, and we extend the bracket complex-bilinearly and the divergence complex-linearly.
Note that $\LL+i\LL$ is closed under brackets, and that a complex field belongs to it exactly when its real and imaginary parts belong to $\LL$.
Put
\begin{equation}\label{eq:zeta}
 \zeta:=x_1+ix_2.
\end{equation}
Then
\begin{equation}\label{eq:zetabracket}
 \{x_3,\zeta\}=-i\zeta.
\end{equation}
The polynomials $\zeta^k$ and $x_3\zeta^{k-1}$ are homogeneous and harmonic, since $\nabla\zeta\cdot\nabla\zeta=\nabla\zeta\cdot\nabla x_3=0$, so their real and imaginary parts restrict to elements of $H_k$ for every $k\geq1$.
Moreover, $\zeta=e^{i\theta}$ on the equator, so $\operatorname{Im}(\zeta^k)=\sin(k\theta)$ is nonconstant.

We now show by induction that $X_{H_k}\subseteq\LL$ for every $k\geq1$; the cases $k=1$ and $k=2$ are already known.
Suppose that $X_{H_k}\subseteq\LL$ for some $k\geq2$.
Then the real and imaginary parts of $X_{\zeta^2}$ and $X_{x_3\zeta^{k-1}}$ lie in $\LL$, so their bracket belongs to $\LL+i\LL$.
Since $\{\zeta^2,x_3\zeta^{k-1}\}=2\zeta\{\zeta,x_3\}\zeta^{k-1}=2i\zeta^{k+1}$, equation~\eqref{eq:XX} gives
\begin{equation}\label{eq:Hrecurrence}
 [X_{\zeta^2},X_{x_3\zeta^{k-1}}]=2iX_{\zeta^{k+1}}.
\end{equation}
Taking real parts of~\eqref{eq:Hrecurrence}, we obtain $-2X_{\operatorname{Im}(\zeta^{k+1})}\in\LL$, and hence $X_{\operatorname{Im}(\zeta^{k+1})}\in\LL$.
This is a nonzero field in $X_{H_{k+1}}$, so $X_{H_{k+1}}\subseteq\LL$ by the first paragraph.
This completes the induction, and by~\eqref{eq:XplusGspace} every Hamiltonian algebraic vector field lies in $\LL$.

We next recover the gradient fields.
For smooth vector fields $U,V$ on $\sph$, we have $\dv[U,V]=U(\dv V)-V(\dv U)$.
Indeed, the Lie derivative of $\omega$ along $U$ is $(\dv U)\omega$, and the Lie derivative along $[U,V]$ is the commutator of the Lie derivatives along $U$ and $V$; applying this commutator to $\omega$, the terms $(\dv U)(\dv V)\omega$ cancel.
Since $\dv G_{x_3}=-2x_3$, $\dv X_{\zeta^k}=0$, and $\{\zeta^k,x_3\}=ik\zeta^k$ by~\eqref{eq:zetabracket}, we obtain, for $k\geq1$,
\begin{equation}
 \dv[X_{\zeta^k},G_{x_3}]=-2X_{\zeta^k}(x_3)=-2ik\zeta^k.
\end{equation}
Now, set $Y_k:=[X_{\zeta^k},G_{x_3}]-\tfrac{2i}{k+1}G_{\zeta^k}$.
By~\eqref{eq:harmonics} and~\eqref{eq:divergence}, applied to the real and imaginary parts of $\zeta^k$, we have $\dv G_{\zeta^k}=-k(k+1)\zeta^k$, and therefore $\dv Y_k=0$.
The real and imaginary parts of $Y_k$ are thus divergence-free algebraic vector fields, so they are Hamiltonian by~\eqref{eq:XplusGspace} and lie in $\LL$ by the induction above.
Hence $Y_k\in\LL+i\LL$.
Since the bracket $[X_{\zeta^k},G_{x_3}]$ also belongs to $\LL+i\LL$, subtracting $Y_k$ and dividing by $2i/(k+1)$ gives $G_{\zeta^k}\in\LL+i\LL$.
Its imaginary part $G_{\operatorname{Im}(\zeta^k)}$ is a nonzero field in $G_{H_k}\cap\LL$, so $G_{H_k}\subseteq\LL$ by the first paragraph.
As this holds for every $k\geq1$, Proposition~\ref{prop:decomposition} gives $\LL=\X(\sph)$.
\end{proof}

\subsection{The two generators $W_1$ and $W_2$}

The vector fields $W_1$ and $X_{x_1}$ are the rotation fields about the $x_3$- and $x_1$-axes, respectively.
The quadratic term $X_{x_1^2-x_2^2}=x\times\operatorname{diag}(2,-2,0)x$ has the form $x\times Dx$ of Euler's rigid-body equation, and the gradient term $G_{x_3}=(0,0,1)-x_3x$ increases height away from the poles, since its $x_3$-component is $1-x_3^2$.
We first show that these three terms of $W_2$ can be obtained individually and then apply Proposition~\ref{prop:criterion} to establish Theorem~\ref{thm:main}.  
We have the following lemma. 

\begin{lemma}\label{lem:separation}
The Lie algebra $\Lie_\R\{W_1,W_2\}$ contains $X_{x_1}$, $X_{x_1^2-x_2^2}$, and $G_{x_3}$.
\end{lemma}

\begin{proof}
Within this proof, let $$A:=X_{x_1}, \quad  B:=X_{x_1^2-x_2^2}, \quad \mbox{and} \quad C:=G_{x_3},$$ 
so that $W_2=A+B+C$. We write $\ad_{W_1}:=[W_1,\cdot]$.
Since $W_1=X_{x_3}$, equations~\eqref{eq:XX} and~\eqref{eq:XG} yield $$\ad_{W_1}A=X_{x_2}, \quad \ad_{W_1}B=4X_{x_1x_2}, \quad \mbox{and} \quad \ad_{W_1}C=0.$$ 
Applying $\ad_{W_1}$ once more, we obtain that 
$$\ad_{W_1}^2A=-A, \quad \ad_{W_1}^2B=-4B, \quad \mbox{and} \quad \ad^2_{W_1}C=0.$$
Thus, $A$, $B$, and $C$ are eigenvectors of $\ad_{W_1}^2$ with respect to the eigenvalues $-1$, $-4$, and $0$, which are pairwise distinct. 

Consequently, $A$, $B$, and $C$ can be obtained by applying polynomials in $\ad_{W_1}^2$ to $W_2$:
\begin{equation}
 \begin{aligned}
 A&=-\tfrac13\bigl(\ad_{W_1}^4W_2+4\ad_{W_1}^2W_2\bigr),\\
 B&=\tfrac1{12}\bigl(\ad_{W_1}^4W_2+\ad_{W_1}^2W_2\bigr),\\
 C&=\tfrac14\bigl(\ad_{W_1}^4W_2+5\ad_{W_1}^2W_2+4W_2\bigr).
 \end{aligned}
\end{equation}
Each right-hand side is a real linear combination of iterated brackets of $W_1$ and $W_2$, so all three fields $A$, $B$, and $C$ belong to $\Lie_\R\{W_1,W_2\}$.
\end{proof}

With the results above, we are now in the position to establish Theorem~\ref{thm:main}.

\begin{proof}[Proof of Theorem~\ref{thm:main}]
Let $\LL:=\Lie_{\R}\{W_1,W_2\}$.
By Lemma~\ref{lem:separation}, $X_{x_1}\in\LL$.
Since $[W_1,X_{x_1}]=X_{x_2}$ and $W_1=X_{x_3}$, every rotation field belongs to $\LL$.
The same lemma gives $X_{x_1^2-x_2^2}\in\LL$ and $G_{x_3}\in\LL$; these fields are nonzero and lie in $X_{H_2}$ and $G_{H_1}$, respectively, since $x_1^2-x_2^2$ is harmonic.
Thus all hypotheses of Proposition~\ref{prop:criterion} hold. We thus conclude that $\LL=\X(\sph)$.
\end{proof}

We end this section by having the following remark: 

\begin{remark}\normalfont
    Both generators $W_1$ and $W_2$ are complete because $\sph$ is compact. It is clear that two is the smallest possible number of generators, since a single vector field $U$ generates only the space $\R U$. 
    For this pair, each term of $W_2$ is necessary.
Deleting $G_{x_3}$ leaves only divergence-free fields.
Deleting $X_{x_1^2-x_2^2}$ keeps all brackets in the six-dimensional Lie algebra spanned by the rotation fields and the gradients of linear functions, by~\eqref{eq:XX}--\eqref{eq:GG}.
If $X_{x_1}$ is deleted, both remaining generators are invariant under the half-turn $(x_1,x_2,x_3)\mapsto(-x_1,-x_2,x_3)$, and so are all their brackets, whereas $X_{x_1}$ changes sign.
Thus the remaining generators cannot produce $X_{x_1}$.
\end{remark}

\section{Proof of Corollary~\ref{cor:control}}\label{sec:PfCor}

We first extend the flow estimate of~\cite[Lemma~8]{Real} to time-dependent vector fields.

\begin{lemma}\label{lem:estimates}
Let $A_t$ and $A_t^j$, $j\ge1$, be time-dependent vector fields on $\sph$ that are smooth in $x$ and piecewise continuous in $t\in[0,T]$ in every $\mathrm{C}^r$ norm, and let $P_A^t$ and $P_{A^j}^t$ denote their flows starting at the identity.
If $A_t^j\to A_t$ in every $\mathrm{C}^r$ norm, uniformly in $t$, then $P_{A^j}^t\to P_A^t$ in every $\mathrm{C}^r$ norm, uniformly in $t\in[0,T]$.
\end{lemma}

\begin{proof}
Fix $r$, and extend maps and vector fields to a neighborhood of $\sph$ by $F(\rho x):=\rho F(x)$ and $U(\rho x):=\rho U(x)$.
The flow $P_A^t$, together with its derivatives up to order $r$, solves the variational equations.
These form an ordinary differential equation on a finite-dimensional space whose right-hand side is built from the derivatives of $A_t$ up to order $r$; this is the jet prolongation used in~\cite[Lemma~8]{Real}.
On a compact neighborhood of the solutions for $A$, the right-hand sides for $A^j$ converge uniformly to the one for $A$, with Lipschitz constants bounded in $j$.
Gronwall's inequality therefore gives uniform convergence of the solutions starting at the derivatives of the identity, which is the claimed $\mathrm{C}^r$ convergence.
\end{proof}

With the lemma above, we establish Corollary~\ref{cor:control}.

\begin{proof}[Proof of Corollary~\ref{cor:control}]
Let $\mathcal G$ be the closure of $\WW$ in the group of smooth diffeomorphisms of $\sph$, with respect to convergence in every $\mathrm{C}^r$ norm.
Since composition and inversion are continuous for this convergence, $\mathcal G$ is a closed subgroup.
We show in turn that $\mathcal G$ contains the flows of algebraic vector fields, the flows of smooth vector fields, and $\Psi$.

Let $U$ be an algebraic vector field and $t\in\R$.
By Theorem~\ref{thm:main}, $tU\in\Lie_\R\{W_1,W_2\}$.
It then follows from~\cite[Lemma~9]{Real}, whose proof uses only that the generators are complete, that $\Phi_U^t$ is a limit of elements of $\WW$ in every $\mathrm{C}^r$ norm.
Thus $\Phi_U^t\in\mathcal G$.

Let $V$ be a smooth vector field on $\sph$.
Polynomial approximation with derivatives~\cite[Lemma~6.9]{AS22}, applied to the three coordinate components, gives polynomial maps $\widetilde U_j$ with $\|\widetilde U_j-V\|_{\mathrm{C}^j}<1/j$.
The algebraic vector fields $U_j(x):=\widetilde U_j(x)-(x\cdot\widetilde U_j(x))x$ converge to $V$ in every $\mathrm{C}^r$ norm, so Lemma~\ref{lem:estimates}, applied to $\pm U_j$ and $\pm V$, gives $\Phi_{U_j}^t\to\Phi_V^t$ for every $t$.
  Since $\Phi_{U_j}^t\in\mathcal G$ by the previous step, we obtain $\Phi_V^t\in\mathcal G$.

Finally, let $\Psi_t$ be a smooth isotopy from the identity to $\Psi$, and let $A_t:=(\partial_t\Psi_t)\circ\Psi_t^{-1}$ be its velocity, so that $\Psi=P_A^1$.
Freezing $A_t$ at the left endpoints of partitions of $[0,1]$ with mesh tending to zero gives piecewise-constant fields that converge to $A_t$ in every $\mathrm{C}^r$ norm, uniformly in $t$.
Their flows at time $1$ are finite compositions of flows of smooth vector fields, so they lie in $\mathcal G$ by the previous step, and they converge to $\Psi$ by Lemma~\ref{lem:estimates}.
Since $\mathcal G$ is closed, $\Psi\in\mathcal G$, which is the statement of Corollary~\ref{cor:control}.
\end{proof}

\section{Conclusions and Outlooks}\label{sec:conclusions}
In this paper, we have shown that the Lie algebra of algebraic vector fields on $\sph$ can be generated by the two vector fields $W_1$ and $W_2$, which we exhibit explicitly in~\eqref{eq:generators}.
In a forthcoming companion paper, we generalize the main result of the paper by showing that, for any $n\geq 3$, the Lie algebra of algebraic vector fields on $S^{n-1}$ is also $2$-generated.

\printbibliography

\end{document}